\documentclass[10pt]{amsart}

\usepackage[latin1]{inputenc}
\usepackage{amsmath}
\usepackage{amsfonts}
\usepackage{amssymb}
\usepackage[mathscr]{eucal}
\usepackage{diagrams}
\usepackage{color,graphicx}
\usepackage{hyperref}
\usepackage[all,2cell]{xy}

\newcommand{\wis}[1]{{\text{\em \usefont{OT1}{cmtt}{m}{n} #1}}}

\newcommand{\C}{\mathbb{C}}

\newcommand{\vtx}[1]{*+[o][F-]{\scriptscriptstyle #1}}
\newcommand{\Oscr}{\mathcal{O}}

\newtheorem{definition}{Definition}

\newtheorem{theorem}{Theorem}
\newtheorem{lemma}{Lemma}
\newtheorem{example}{Example}

\title{Representation stacks, D-branes and noncommutative geometry}
\author{Lieven Le Bruyn} 
\address{Department of Mathematics, University of Antwerp \\ 
 Middelheimlaan 1, B-2020 Antwerp (Belgium) \\ {\tt lieven.lebruyn@ua.ac.be}}

\begin{document}
\sloppy
\UseAllTwocells

\maketitle

\begin{abstract}
In this note we prove that the $\wis{spec}(C)$-points of the representation Artin-stack $[\wis{rep}_n R/PGL_n]$ of $n$-dimensional representations of an affine $\C$-algebra $R$ correspond to $\C$-algebra morphisms $R \rTo A_n$ where $A_n$ is an Azumaya algebra of degree $n$ over $C$. We connect this to the theory of D-branes and Azumaya noncommutative geometry, developed by Chien-Hao Liu and Shing-Tung Yau in a series of papers \cite{LY1}-\cite{LY6}. 
\end{abstract}

\section{Representation stacks}

Throughout, all algebras will be associative, unital $\C$-algebras which are finitely generated, that is, have a presentation
\[
R = \frac{\C \langle x_1,\hdots,x_k \rangle}{(p_i(x_1,\hdots,x_k)~|~ i \in I)} \]
With $\wis{rep}_n~R$ we will denote the affine scheme of all $n$-dimensional representations of $R$. Its coordinate ring is the quotient of the polynomial algebra in all entries of $k$ generic $n \times n$ matrices
\[
X_i = \begin{bmatrix} x_{11}(i) & \hdots & x_{1n}(i) \\
\vdots & & \vdots \\
x_{n1}(i) & \hdots & x_{nn}(i) \end{bmatrix} \]
modulo the ideal generated by all the matrix-entries of the $n \times n$ matrices $p_i(X_1,\hdots,X_k)$ for all $i \in I$. That is,
\[
\Oscr(\wis{rep}_n~R) = \frac{\C[x_{ij}(l)~:~1 \leq i,j \leq n, 1 \leq l \leq k]}{(p_i(X_1,\hdots,X_k)_{uv}~:~1 \leq u,v \leq n,~i \in I)} \]
The affine group scheme $GL_n$ (or rather $PGL_n$) acts via simultaneous conjugation on the generic matrices $X_i$ and hence on the affine scheme $\wis{rep}_n~R$, its orbits corresponding to isomorphism classes of $n$-dimensional representations of $R$. It is well-known that the GIT-quotient $\wis{rep}_n R/PGL_n$, that is the affine scheme corresponding to the ring of polynomial invariants $\Oscr(\wis{rep}_n R)^{PGL_n}$, classifies isomorphism classes of {\em semi-simple} $n$-dimensional representations of $R$, see \cite{ProcesiCH} or \cite[Chp. 2]{LBbook}.

In order to investigate all orbits, one considers the $n$-dimensional {\em representation stack} $[\wis{rep}_n R / PGL_n ]$. By this we mean (following \cite{deJong}) the category with  objects all triples $(Z,P,\phi)$ where $Z$ a scheme over $\C$, $\pi_P : P \rOnto Z$ a $PGL_n$-torsor in the \'etale topology
\[
\xymatrix{
P \ar[rr]^{\phi} \ar[d]^{\pi_P} & & \wis{rep}_n R \\
Z & & }
\]
and with $\phi$ a $PGL_n$-equivariant map. Morphisms in this category are pairs $(f,h) : (Z,P,\phi) \rTo (Z',P',\phi')$ where $f : Z \rTo Z'$ is a morphism
\[
\xymatrix{P \ar[rr]^{\phi} \ar[rd]^h \ar[d]^{\pi_P} & & \wis{rep}_n R \\
Z \ar[rd]_f & P' \ar[ru]_{\phi'} \ar[d]^{\pi_{P'}} & \\
& Z' &}
\]
and $h$ is a $PGL_{n}$-equivariant map such that $P \simeq Z \times_{Z'} P'$ and $\phi = \phi' \circ h$. 

The full subcategory consisting of all triples $(Z,P,\phi)$ with fixed $Z$ are called the $Z$-points of the stack and are denoted by $[ \wis{rep}_{n} R / PGL_n ](Z)$. Observe that this subcategory is a groupoid, that is, all its morphisms are isomorphisms.

\begin{theorem} For a commutative affine $\C$-algebra $C$, the $\wis{spec}(C)$-points of the representation stack $[\wis{rep}_n R/PGL_n]$ are in natural one-to-one correspondence with $\C$-algebra morphisms $R \rTo A_n$ where $A_n$ is an Azumaya algebra of degree $n$ with center $C$.
\end{theorem}

Before we can prove this result, we need to recall some facts on Cayley-Hamilton algebras, see \cite{ProcesiCH} and \cite{LBbook}, and on Azumaya algebras, see \cite{DeMeyerIngraham} and \cite{KnusOjanguren}.

 Every $\C$-algebra $R$ has a universal trace map $n_R~:~R \rTo R/[R,R]_v$ where $[R,R]_v$ is the sub-vectorspace of $R$ spanned by all commutators $[r,s]=rs-sr$ in $R$. It allows us to define the {\em necklace functor}
 \[
 \oint~:~\wis{alg} \rTo \wis{commalg} \]
 which assign to any $\C$-algebra $R$ its necklace algebra $\oint R = Sym(\tfrac{R}{[R,R]_v})$ where for any $\C$-vectorspace $V$ we denote by $Sym(V)$ the symmetric algebra on $V$.
 
 With $\wis{alg@}$ we denote the category of $\C$-algebras {\em with trace}. That is, a $\C$-algebra $R$ belongs to $\wis{alg@}$ if it has a linear map $tr~:~R \rTo R$ satisfying the following properties for all $r,s \in R$
 \[
 \begin{cases}
 tr(r)s=str(r) \\
 tr(rs)=tr(sr) \\
 tr(tr(r)s) = tr(r)tr(s)
 \end{cases}
 \]
 In particular, it follows that the image of the trace map is contained in the center $Z(R)$. Morphisms in $\wis{alg@}$ are trace preserving $\C$-algebra morphisms. The forgetful functor $\wis{alg@} \rInto \wis{alg}$ has a left adjoint, called the {\em trace algebra functor}
 \[
\int~:~ \wis{alg} \rTo \wis{alg@} \qquad R \mapsto \int R = \oint R \otimes R \]
with the trace map on $\int R$ defined by $tr(c \otimes r) = c n_R(r) \otimes 1$, see \cite{LBbook}. That is, for any $\C$-algebra $R$ and any $\C$-algebra with trace $A$,  there is a natural one-to-one correspondence 
\[
Hom_{\wis{alg@}}(\int R,A) \rTo Hom_{\wis{alg}}(R,A) \qquad \text{given by} \qquad f \mapsto f \circ n_R \]
Let  $M \in M_n(\C)$ be a diagonalizable matrix with distinct eigenvalues $\{ \lambda_1,\hdots,\lambda_n \}$, then the characteristic polynomial has as coefficients the elementary symmetric functions $\sigma_i = \sigma_i(\lambda_1,\hdots,\lambda_n)$. Another generating set of the ring of symmetric functions is given by the power sums $s_i = \lambda_1^i + \hdots + \lambda_n^i$, hence there exist uniquely determined polynomials $p_i$ such that 
\[
\sigma_i(\lambda_1,\hdots,\lambda_n) = p_i(s_1,\hdots,s_n) = p_i(Tr(M),\hdots,Tr(M^i)) \]
because $s_i = Tr(M^i)$. 

This allows us to define for an algebra with trace $(R,tr_R) \in \wis{alg@}$ the formal Cayley-Hamilton polynomial of degree $n$ for all elements $a \in R$ by
\[
\chi_a^{(n)}(t) = t^n + p_1(tr_R(a)) t^{n-1} + p_2(tr_R(a),tr_R(a^2)) t^{n-2} + \hdots + p_n(tr_R(a),\hdots,tr_R(a^n)) \]
With $\wis{alg@n}$ we denote the category of all Cayley-Hamilton algebras of degree n, that is, having as its objects algebras with trace map $(R,tr_R)$ satisfying $tr_R(1)=n$ and $\chi_a^{(n)}(a) = 0$ for all $a \in R$, and, trace preserving $\C$-algebra maps as morphisms. We have functors
\[
\begin{cases}
\int_n~:~\wis{alg} \rTo \wis{alg@n} &\qquad R \mapsto \frac{\int R}{(tr_R(1)-n, \chi_a^{(n)}(a)~\forall a \in \int R)} \\
\oint_n~:~\wis{alg} \rTo \wis{commalg} &\qquad R \mapsto tr_{\int_n R}(\int_n R)
\end{cases}
\]
The main structural results on the trace algebras $\int_n R$ and $\oint_n R$ and their invariant theoretic interpretation are summarized in the next theorem, due to Claudio Procesi. For a proof and more details the reader is referred to \cite[Chp. 2]{LBbook} or to the original paper \cite{ProcesiCH}.

\begin{theorem}[Procesi] \label{Procesi} Let $R$ be a finitely generated $\C$-algebra. Then, with notations as before
\begin{enumerate}
\item{$\oint_n R$ is an affine commutative $\C$-algebra}
\item{$\int_n R$ is a finitely generated module over $\oint_n R$}
\item{$\oint_n R$ is the coordinate ring of the GIT-quotient scheme $\wis{rep}_n~R / PGL_n$}
\item{$\int_n R$ is the ring of all $PGL_n$-equivariant maps $\wis{rep}_n~R \rTo M_n(\C)$}
\end{enumerate}
\end{theorem}

{\em Azumaya algebras} form an important class of Cayley-Hamilton algebras. For a $C$-algebra $R$ denote its enveloping algebra $R^e = R \otimes_C R^{op}$, where $R^{op}$ is the $C$-algebra with opposite multiplication map. There is a natural $C$-algebra morphism
\[
j_R~:~R^e = R \otimes_C R^{op} \rTo End_C(R) \qquad j_R(a \otimes b)r = arb \]

\begin{definition} A $C$-algebra $A$ is called an {\em Azumaya algebra} iff $A$ is a finitely generated projective $C$-module and the map $j_A$ is an isomorphism of $C$-algebras.
\end{definition}

It follows that the center $Z(A)$ of $A$ is equal to $C$ and that $A/\mathfrak{m}A \simeq M_n(\C)$ for every maximal ideal $\mathfrak{m}$ of $C$ where $n^2$ is the local rank of $A$ at $\mathfrak{m}$. If the local rank is constant  and equal to $n^2$ we say that $A$ is an Azumaya algebra {\em of degree $n$}. We recall the geometry associated to $C$-Azumaya algebras of degree $n$.

\begin{lemma} $C$-Azumaya algebras of degree $n$ are classified up to $C$-algebra isomorphism by the \'etale cohomology group
$H^1_{et}(\wis{spec}(C),PGL_n)$, that is, $C$-Azumaya algebras of degree $n$ correspond to principal $PGL_n$-fibrations over $\wis{spec}(C)$.
\end{lemma}

\begin{proof} It is well known (see for example \cite{KnusOjanguren}) that $A$ is a $C$-Azumaya algebra of degree $n$ if and only if there exist \'etale extensions $C \rTo D$ splitting $A$, that is, such that $A \otimes_C D \simeq M_n(D)$. As $Aut(M_n(D)) = PGL_n(D)$, the isoclasses of such algebras are classified by the claimed cohomology group, see for example \cite{Milne}.
\end{proof}

The principal $PGL_n$-fibration corresponding to the degree $n$ $C$-Azumaya algebra $A$ is the representation scheme
\[
\wis{rep}_n A \rOnto \wis{rep}_n A / PGL_n = \wis{spec}(C) \]
as all finite dimensional simple $A$-representation have to be $n$-dimensional because $A/\mathfrak{m}A = M_n(\C)$. By \'etale descent, $A$ has a reduced trace map with $tr(A) = C$, and hence we deduce from theorem~\ref{Procesi} that $A \in \wis{alg@n}$ and 
\[
\int_n A = A \qquad \text{and} \qquad \oint_n A = C \]
We have now all the tools needed to prove theorem~1.

\vskip 4mm
\noindent
{\bf Proof of theorem~1 : } Every $\C$-algebra morphism $\beta : R \rTo A_n$ induces a $PGL_n$-equivariant map
\[
\beta^*~:~\wis{rep}_n A_n \rTo \wis{rep}_n R \]
by composition. By the above remarks we know that the GIT-quotient
\[
\wis{rep}_n A_n \rOnto^{\pi} \wis{rep}_n~A_n/PGL_n = \wis{spec}(C) \]
is a $PGL_n$-torsor. Therefore $\beta : R \rTo A_n$ determines the $\wis{spec}(C)$-point of the representation stack
\[
\xymatrix{ \wis{rep}_n A_n  \ar[d]^{\pi} \ar[r]^{\beta^*} & \wis{rep}_n R \\
\wis{spec}(C) &}
\]
Conversely, $PGL_n$-torsors over $\wis{spec}(C)$ are classified by the pointed set
\[
H^1_{et}(\wis{spec}(C),PGL_n) \]
which also classifies the isomorphism classes of Azumaya algebras of degree $n$ with center $C$, see for example \cite[p. 134]{Milne}. Hence, any $\wis{spec}(C)$-point of the representation stack $[\wis{rep}_n R/PGL_n]$ is of the form
\[
\xymatrix{ \wis{rep}_n A_n  \ar[d]^{\pi} \ar[r]^{\phi} & \wis{rep}_n R \\
\wis{spec}(C) &}
\]
Taking $PGL_n$-equivariant maps to $M_n(\C)$ on both sides of the $PGL_n$-equivariant map $\phi$ gives us by theorem~\ref{Procesi} and the remarks above a trace preserving algebra morphism
\[
\int_n R \rTo^{\phi_*} \int_n A_n = A_n \]
and composing this with the universal morphism $R \rTo \int_n R$ we obtain the desired $\C$-algebra morphism $R \rTo^{\beta} A_n$ which induces $\phi = \beta^*$.

\section{Algebraic D-branes}

In this section we relate the above to the description of D-branes via Azumaya noncommutative geometry as developed by Chian-Hao Liu and Shing-Tung Yau in \cite{LY1}-\cite{LY6}.

Let the affine variety $X$ be an affine open piece of (a spatial slice of) space-time and let $Z$ be a closed subscheme giving the boundary conditions for the endpoints of open strings (a D-brane). Wrapping one brane around $Z$ gives a $\C^*$-bundle (or a $U(1)$-symmetry) and the embedding of the D-brane in space-time $X$ corresponds to the quotient map of the corresponding coordinate rings
\[
\beta_1 : \Oscr(X) \rOnto \Oscr(Z) \]
However, if there are $n$ branes wrapped around $Z$, then $Z$ comes equipped with a $GL_n$-bundle $P$ (or a $U(n)$-symmetry) and the so called {\em Polchinski Ansatz}, see for example \cite{LY1}, asserts that the embedding of the $n$-stack of branes in space time is now governed by a $\C$-algebra morphism
\[
\beta  : \Oscr(X) \rTo End_{\Oscr(Z)}(P) \]
where $P$ is the so-called Chan-Patton bundle on $Z$. Observe that $End_{\Oscr(Z)}(P)$ is a trivial Azumaya algebra of degree $n$ with center $\Oscr(Z)$ and hence the embedding morphism $\beta$ gives a $Z$-point of the representation stack $[\wis{rep}_n \Oscr(X)/PGL_n]$.

 In more general situations, for example when the $B$-field is turned on, one may replace the trivial Azumaya algebra $End_{\Oscr(Z)}(P)$ by a non-trivial Azumaya algebra $A_n$ of degree $n$ with center $\Oscr(Z)$. Sometimes, one even allows for a noncommutative space time $R$. Motivated by this, we formalize D-branes in a purely algebraic context.

\begin{definition} For $R$ an affine $\C$-algebra and $Z$ an affine scheme we will call a $Z$-point of the representation stack
\[
\beta \in [\wis{rep}_n R/PGL_n](Z) \]
an {\em algebraic D-brane} of degree $n$ on $Z$.
\end{definition}

The upshot of this interpretation of a D-brane embedding as a $Z$-point in the representation stack $[\wis{rep}_n \Oscr(X)/PGL_n]$ (or more generally $[\wis{rep}_n R/PGL_n]$) is that one expects good properties of D-branes whenever the representation stack is smooth, in particular when $\wis{rep}_n R$ is a smooth affine variety.

This connects the study of D-branes to that of Cayley-smooth orders and formally smooth algebras as described in detail in the book \cite{LBbook}. Observe that several of the examples worked out by Liu and Yau, such as the case of curves in \cite[\S 4]{LY1} and \cite{LY2}, or the case of the conifold algebra \cite{LY4} (see also \cite{LBSymens}) fall in this framework. 

The dynamical aspect of D-branes, that is their Higgsing and un-Higgsing behavior as in \cite{LY1} can also be formalized purely algebraically by enhancing the category $\wis{alg}$ of all $\C$-algebras with a specific class of 2-morphisms.

\begin{definition} 
A {\em 2-morphism} between two  $\C$-algebra morphisms $f$ and $g$
\[
\xymatrix{ R \rtwocell^f_g{\alpha} & S} \]
is a $Z(S)$-algebra mono-morphism between the centralizers in $S$ of the images
\[
\alpha~:~C_S(Im(f)) \rInto C_S(Im(g)) \]
If such a 2-morphism exists we say that $f$ degenerates to $g$ or, equivalently, that $g$ deforms to $f$. 
\end{definition}

In the case of algebraic D-branes, that is when $S=A_n$ is an Azumaya algebra of degree $n$ with center $C=\Oscr(Z)$, these 2-morphisms correspond to the notions of 'Higgsing' and 'un-Higgsing' as in \cite{LY1}. 
Here, the idea is that the Lie-algebra structure of the centralizer $C_A(Im(\beta))$ can be interpreted as the gauge-symmetry group of the D-brane. Higgsing corresponds to symmetry-breaking, that is the gauge-group becomes smaller or, in algebraic terms, a 2-morphism deformation. Likewise, un-Higgsing corresponds to a 2-morphism degeneration of the algebraic D-brane.

\begin{example} There are two extremal cases of algebraic D-branes. The 'maximal' ones coming from epimorphisms
\[
R \rOnto^{\beta} A \]
If $C$ is the center of $A$ such a D-brane determines (and is determined by) a $\wis{spec}(C)$-family of simple $n$-dimensional representations of $R$.
At the other extreme, 'minimal' algebraic D-branes are given by a composition
\[
\beta~:~R \rOnto^{\pi} C \rInto^i A \]
where $C$ is a commutative quotient of $R$ and $A$ is an Azumaya algebra over $C$. Such D-branes essentially determine a $\wis{spec}(C)$-family of one-dimensional representations of $R$.

For a maximal D-brane $\beta_{max} : R \rOnto A$ we have $C_A(Im(\beta_{max})) = Z(A)$ whence for any other algebraic D-brane $g : R \rTo A$ we have a two-cell and hence a degeneration
\[
\xymatrix{ R \rtwocell^{\beta_{max}}_g{\alpha} & A} \]
Likewise, if $\beta_{min} : R \rOnto C \rInto A$ is a minimal D-brane we have $C_A(Im(\beta_{min})) = A$ whence for any other D-brane we have a two-cell and corresponding deformation
\[
\xymatrix{ R \rtwocell^g_{\beta_{min}}{\alpha} & A} \]
\end{example}

Note however that these 2-morphisms do {\em not} equip $\wis{alg}$ with a 2-category structure as defined for example in \cite[XII.3]{MacLane}. Indeed, whereas we have an obvious vertical composition of 2-morphisms there is in general no horizontal composition of 2-morphisms.

Still, these 2-morphisms impose a natural compatibility structure on families of algebraic D-branes. In string-theory one often considers the limit $n \rightarrow \infty$ of n stacks of D-branes located at a subscheme $Z \rInto X$, that is, a family of algebra morphisms $\beta_n : \C[X] \rTo A_n$ where $A_n$ is a degree $n$ Azumaya algebra with center $\C[Z]$. Our ringtheoretical considerations suggest one needs to impose a multiplicative compatibility requirement on such families.

\begin{definition} A family of algebraic D-branes
\[
\beta_n~:~R \rTo A_n \]
such that $Z(A_n)=C$ for all $n$ is said to be {\em compatible} if for every $n | m$ we have a $C$-monomorphism $i_{nm} : A_n \rInto A_m$ and a corresponding 2-morphism
\[
\xymatrix{ R \rtwocell^{\beta_m}_{i_{nm} \circ \beta_n}{\alpha} &  A_m} \]
That is, the algebraic D-brane $\beta_m$ corresponding to the m stack of D-branes at $\wis{spec}(C)$ is a Higgsing of all D-branes $R \rTo A_n \rInto A_m$ for all divisors $n$ of $m$.
\end{definition}

We will give an example of a non-trivial example of a family of algebraic D-branes which are all neither minimal nor maximal.
Consider the 2-dimensional torus $\C[s^{\pm},t^{\pm}]$ and a primitive $n$-th root of unity $q_n = \sqrt[n]{1}$. The {\em quantum torus} is the non-commutative algebra generated by two elements $U_n$ and $V_n$ (and their inverses) satisfying the relations
\[
V_n U_n = q_n U_n V_n \quad U_n^n=s \quad V_n^n=t \]
and we will denote this algebra as $\C_{q_n}[U_n^{\pm},V_n^{\pm}]$. The center of this algebra is easily seen to be $\C[s^{\pm},t^{\pm}]$ and   in fact $\C_{q_n}[U_n^{\pm},V_n^{\pm}]$ is an Azumaya algebra over it of degree $n$. When $m=n.k$ there are obvious embeddings of $\C[s^{\pm},t^{\pm}]$-algebras
\[
i_{m,n}~:~\C_{q_n}[U^{\pm}_n,V^{\pm}_n] \rInto \C_{q_m}[U_m^{\pm},V_m^{\pm}] \qquad \begin{cases} U_n \rTo U_m^k \\ V_n \rTo V_m^k \end{cases} \]

\begin{lemma} Consider the coordinate ring of $GL_2 = \begin{bmatrix} s & u \\ v & t \end{bmatrix}$
\[
\Oscr(GL_2) = \C[s,t,u,v,(st-uv)^{-1}] \]
 then the $\C$-algebra morphisms 
 $\beta_n : \Oscr(GL_2) \rTo \C_{q_n}[U_n^{\pm},V_n^{\pm}]$ defined by
 \[
 \beta=(\beta_n)~:~\begin{cases} u \mapsto 0 \\ v \mapsto 0 \\ s \mapsto s \\ t \mapsto V_n \end{cases} \]
 is a compatible family of algebraic D-branes on the maximal torus $T_2 \rInto GL_2$.
\end{lemma}

\begin{proof} Clearly the maps  $\beta_n$ are $\C$-algebra morphisms from $\Oscr(GL_2)$ to the Azumaya algebra $\C_{q_n}[U_n^{\pm},V_n^{\pm}]$ whence they are algebraic D-branes. Remains for every $m = n.k$ to compare the centralizers of the images at level $m$ with those of the image with the inclusion $i_{m,n}$. One easily verifies that
\[
Im(\beta_m) = \C[s^{\pm},V_m^{\pm}] \quad \text{whereas} \quad Im(i_{m,n} \circ \beta_n) = \C[s^{\pm},V_m^{\pm k}] \]
The centralizer of $Im(\beta_m)$ is equal to itself, whereas 
\[
C(Im(i_{m,n} \circ \beta_n) = \C[s^{\pm},V_m^{\pm k}] \otimes_{\C[s^{\pm},t^{\pm}]} \C_{q_k}[U_m^{\pm n},V_m^{\pm n}] \]
and so these form indeed a compatible family of algebraic D-branes on $T_2$.
\end{proof}

\section{Azumaya noncommutative geometry}

The main problem studied in the papers \cite{LY1}-\cite{LY6} by Chien-Hao Liu and Shing-Tung Yau is to associate some noncommutative geometry to the morphism $\beta  : \Oscr(X) \rTo A_n$ (or more generally $\beta  : R \rTo A_n$) describing the embedding of the n-stack of D-branes in space-time. That is, one wants to associate geometric objects (such as varieties, topological spaces and sheaves on these)  that will allow us to reconstruct $\beta $ from the geometric data. 

\subsection{Noncommutative structure sheaves}

We will rephrase Liu and Yau's Azumaya noncommutative geometry, based on 'surrogates' and the 'noncommutative cloud', in classical noncommutative algebraic geometry of pi-algebras, as developed by Fred Van Oystaeyen and Alain Verschoren \cite{LNM887} in the early 80-ties.

Let $S$ be a finitely generated Noetherian $\C$-algebra satisfying all polynomial identities of degree $n$. Observe that $\int_n R$, $A_n$ and its subalgebras 
\[
A_c=Im(\beta ) C \qquad \text{and} \qquad  A_e=Im(\beta ) C_A(Im(\beta )) \]
 all satify these requirements. With $\wis{spec}(S)$ we denote the set of all twosided prime ideals of $S$ which becomes a topological space by equipping it with the Zariski topology having as typical closed sets $\mathbb{V}(I) = \{ P \in \wis{spec}(S)~|~I \subset P \}$ for $I$ a twosided ideal of $S$. However, this twosided prime spectrum is not necessarily functorial, that is, if $S \rTo^f T$ is a $\C$-algebra morphism between suitable algebras, then $f^{-1}(P)$ for a twosided prime ideal $P \triangleleft T$ does not have to be a prime ideal in $S$. 
Still, the twosided prime spectrum is functorial whenever $S \rTo^f T$ is an {\em extension}, that is, if $T = Im(f) C_T(Im(f))$ by \cite[Thm. II.6.5]{ProcesiBook}.

In \cite[V.3]{LNM887} Van Oystaeyen and Verschoren constructed a noncommutative structure sheaf $\mathcal{O}^{bi}_S$ on $\wis{spec}(S)$ by means of localization in the category of $S$-bimodules. The main properties of this construction are \cite[Thm. V.3.17 and Thm. V.3.36]{LNM887}
\begin{itemize}
\item{One recovers the algebra back from the global sections $S = \Gamma(\wis{spec}(S),\mathcal{O}^{bi}_S)$}
\item{Every extension $S \rTo^f T$ defines a morphism of ringed spaces
\[
(\wis{spec}(T), \mathcal{O}^{bi}_T) \rTo (\wis{spec}(S),\mathcal{O}_S^{bi}) \]}
\end{itemize}
In fact, the results are stated there only for prime Noetherian rings satisfying all polynomial identities of degree $n$ but extend to the case of interest here as our rings are finite modules over an affine center and hence we have enough central localizations. For more details, see \cite{FredBook}.

An algebraic $D$-brane $R \rTo^{\beta} A$ of degree $n$ is not necessarily an extension, but the morphisms below defined by it are :
\[
\int_n R \rTo^{\beta } A_c = Im(\beta ) C \quad \text{and} \quad \int_n R \rTo^{\beta } A_e = Im(\beta ) C_A(Im(\beta )) \]

\begin{theorem}
Associate to the D-brane either of these noncommutative geometric data
\[
(\wis{spec}(A_{\ast}), \mathcal{O}^{bi}_{A_{\ast}}) \rTo (\wis{spec}(\int_n R),\mathcal{O}^{bi}_{\int_n R}) \]
It contains enough information to reconstruct $\beta ~:~\int_n R \rTo A$ by taking global sections, and hence by compositing with the universal morphism $R \rTo \int_n R$ also the D-brane.
\end{theorem}

The fact that noncommutative prime spectra and bimodule structure sheaves behave only functorial with respect to $\C$-algebra morphisms which are extensions, explains the notion of a 'noncommutative cloud' in \cite{LY1}. In our language, the noncommutative cloud of a $C$-Azumaya algebra $A$ is the set
\[
\wis{cloud}(A) = \bigsqcup_{A_{\ast}} \wis{spec}(A_{\ast}) \]
where the disjoint union of noncommutative prime spectra is taken over all subalgebras $A_{\ast}$ satisfying
\[
C \subset Z(A_{\ast}) \subset A_{\ast} \subset A \]
and such subalgebras are called 'surrogates' in \cite{LY1}. The point being that we can now associate to an algebraic D-brane of degree $n$, $R \rTo^{\beta} A$ the partially defined map which is well defined because $\int_n R \rTo A_e$ is an extension
\[
\wis{cloud}(A) \lInto \wis{spec}(A_e) \rTo \wis{spec}(\int_n R) \]
and by adorning $\wis{cloud}(A)$ componentswise with noncommutative structure sheaves $\mathcal{O}^{bi}_{A_{\ast}}$, this partially defined map contains enough information to reconstruct the D-brane.

\subsection{Noncommutative thin schemes}

 In \cite[\S I.2]{KontSoib} Maxim Kontsevich and Yan Soibelman define a {\em noncommutative thin scheme} to be a covariant functor commuting with finite projective limits
\[
\wis{X}~:~\wis{fd-alg} \rTo \wis{sets} \]
from the category $\wis{fd-alg}$ of all finite dimensional $\C$-algebras to the category $\wis{sets}$ of all sets. They prove \cite[Thm. 2.1.1]{KontSoib} that every noncommutative thin scheme is represented by a $\C$-coalgebra. That is, there is a $\C$-coalgebra $C_{\wis{X}}$ associated to the noncommutative thin scheme $\wis{X}$ having the property that there is a natural one-to-one correspondence
\[
\wis{X}(B) = \wis{alg}(B,C_{\wis{X}}^*) \]
for every finite dimensional $\C$-algebra $B$. Here, $C_{\wis{X}}^*$ is the dual algebra of all linear functionals on $C_{\wis{X}}$. $C_{\wis{X}}$ is called the {\em coalgebra of distributions} on $\wis{X}$ and the {\em noncommutative algebra of functions} on the thin scheme $\wis{X}$ is defined to be $\C[\wis{X}] = C_{\wis{X}}^*$.
For a $\C$-algebra $R$, it is not true in general that the linear functionals $R^*$ are a coalgebra, but the {\em dual coalgebra} $R^o$ is, where 
\[
R^0 = \{ f \in R^*~|~Ker(f)~\text{contains a twosided ideal of finite codimension}~\} \]
Kostant duality, see for example \cite[Thm. 6.0.5]{Sweedler}, asserts that the functors
\[
\xymatrix{\wis{alg} \ar@/^/[rr]^{\circ} & & \wis{coalg} \ar@/^/[ll]^{\ast}} \]
are adjoint. That is, for every $\C$-algebra $R$ and $\C$-coalgebra $C$ there is a natural one-to-one correspondence between  the homomorphisms
\[
\wis{alg}(R,C^*) = \wis{coalg}(C,A^{\circ}) \]
For an affine $\C$-algebra $R$ we define the contravariant functor 
\[
\wis{rep}_R~:~\wis{fd-coalg} \rTo \wis{sets} \qquad C \mapsto \wis{alg}(R,C^*) \]
describing the finite dimensional representations of $R$, see \cite[Example 2.1.9]{KontSoib}. As taking the linear dual restricts Koszul duality to an aniti-equivalence between the categories $\wis{fd-alg}$ and $\wis{fd-coalg}$, we can describe $\wis{rep}_R$ as the noncommutative thin scheme  represented by the dual coalgebra $R^{\circ}$ as
\[
\wis{rep}_R~:~\wis{fd-alg} \rTo \wis{sets} \qquad B=C^* \mapsto \wis{alg}(R,B) = \wis{coalg}(C,R^{\circ}) \]

\begin{definition} The {\em noncommutative affine scheme} $\wis{rep}_R$ is the noncommutative thin scheme represented by the dual coalgebra $R^{\circ}$ of the affine $\C$-algebra $R$.
\end{definition}

By \cite[Lemma 6.0.1]{Sweedler} for every $\C$-algebra morphism $f \in \wis{alg}(R,B)$, the dual map determines a $\C$-coalgebra morphism $f^* \in \wis{coalg}(B^{\circ},R^{\circ})$. In particular, to an algebraic D-brane $\beta : R \rTo A$ we can associate a morphism between their noncommutative affine schemes
\[
\beta^*~:~\wis{rep}_A \rTo \wis{rep}_R \]
determined by the coalgebra map $\beta^* : A^{\circ} \rTo R^{\circ}$. 

\begin{theorem} \label{reconstructbrane} The morphism $\beta^* : \wis{rep}_A \rTo \wis{rep}_R$ between the noncommutative thin schemes allows to reconstruct the algebraic D-brane $\beta : R \rTo A$.
\end{theorem}

In order to prove this we need to describe the dual coalgebras $A^o$ and $R^o$.
Recall that a coalgebra $D$ is said to be simple if it has no proper nontrivial sub-coalgebras. Every simple $\C$-coalgebra is finite dimensional and as $D^*$ is a simple $\C$-algebra, we have that $D \simeq M_n(\C)^*$, the full matrix coalgebra that is, $\sum_{i,j} \C e_{ij}$ with
\[
\Delta(e_{ij}) = \sum_{k=1}^n e_{ik} \otimes e_{kj} \quad \text{and} \quad \epsilon(e_{ij}) = \delta_{ij} \]
The coradical $corad(C)$ of a coalgebra $C$ is the direct sum of all simple subcoalgebras of $C$. It follows from Kostant duality that for any affine $\C$-algebra $R$ we have
\[
corad(R^o) = \oplus_{S \in \wis{simp} R} M_n(\C)^*_S \]
where $\wis{simp} R$ is the set of isomorphism classes of finite dimensional simple $R$-representations and the factor of $corad(R^o)$ corresponding to a simple representation $S$ is isomorphic to the matrix coalgebra $M_n(\C)^*$ if $dim(S)=n$.

In algebra, one can resize idempotents by Morita-theory and hence replace full matrices by the basefield.  In coalgebra-theory there is an analogous duality known as {\em Takeuchi equivalence}, see \cite{Takeuchi}.
The isotypical decomposition of $corad(R^o)$ as an $R^o$-comodule is of the form $\oplus_S C_S^{\oplus n_S}$, the sum again taken over all simple finite dimensional $R$-representations. Take the $R^o$-comodule
$E= \oplus_S C_S$ and its {\em coendomorphism coalgebra} 
\[
R^{\dagger} = coend^{R^o}(E) \]
then Takeuchi-equivalence (see for example \cite[\S 4, \S 5]{Chin} and the references contained in this paper for more details) asserts that $R^o$ is Takeuchi-equivalent to the coalgebra $R^{\dagger}$ which is {\em pointed}, that is, $corad(R^{\dagger}) = \C~\wis{simp}(R) = \oplus_S \C g_S$ with one {\em group-like} element $g_S$ for every simple finite dimensional $R$-representation. Remains to describe the structure of the full  basic coalgebra $R^{\dagger}$.

\begin{example} \label{exampleAzu} For the affine Azumaya algebra $A$ of degree $n$ over its affine center $C$, we know that all finite dimensional simple $A$-representations are $n$-dimensional and are parametrized by the maximal ideals $\mathfrak{m} \in \wis{max}(C)$. That is, $corad(A^o) = \oplus_{\mathfrak{m}} M_n(\C)^*$ and $A^o$ is Takeuchi-equivalent to the pointed coalgebra
\[
A^{\dagger} = C^o \quad \text{with} \quad corad(A^{\dagger}) = \oplus_{\mathfrak{m}} \C g_{\mathfrak{m}} \]
By \cite[Prop. 8.0.7]{Sweedler} we know that any cocommutative pointed coalgebra is the direct sum of its pointed irreducible components (at the algebra level this says that a commutative semi-local algebra is the direct sum of local algebras). Therefore,
\[
A^{\dagger} = C^0 = \oplus_{\mathfrak{m}} C^0_{\mathfrak{m}} \]
where $C^0_{\mathfrak{m}}$ is a pointed irreducible cocommutative coalgebra and as such is a sub-coalgebra of the enveloping coalgebra of the abelian Lie algebra on the Zariski tangent space $(\mathfrak{m}/\mathfrak{m}^2)^*$. That is, we recover the maximal spectrum $\wis{max}(C)$ of the center $C$ from $A^{\dagger}$. But then, the dual algebra
\[
A^{\dagger \ast} = C^{\circ \ast} = \prod_{\mathfrak{m}} \hat{C}_{\mathfrak{m}} \]
the direct sum of the completions of $C$ at all maximal ideals $\mathfrak{m}$. Also, the double dual algebra
\[
A^{o \ast} = \prod_{\mathfrak{m}} M_n(\hat{C}_{\mathfrak{m}}) \]
\end{example}

For a general affine noncommutative $\C$-algebra $R$, the description of the pointed coalgebra $R^{\dagger}$ is more complicated as there can be non-trivial extensions between non-isomorphic finite dimensional simple $R$-representations (note that this does not happen for Azumaya algebras).

For a (possibly infinite) quiver $\vec{Q}$ we define the {\em path coalgebra} $\C \vec{Q}$ to be the vectorspace $\oplus_p \C p$ with basis all oriented paths $p$ in the quiver $\vec{Q}$ (including those of length zero, corresponding to the vertices) and with structural maps induced by
\[
\Delta(p) = \sum_{p=p'p"} p' \otimes p" \qquad \text{and} \qquad \epsilon(p) = \delta_{0,l(p)} \]
where $p'p"$ denotes the concatenation of the oriented paths $p'$ and $p"$ and where $l(p)$ denotes the length of the path $p$. Hence, every vertex $v$ is a group-like element and for an arrow $\xymatrix{\vtx{v} \ar[r]^a & \vtx{w}}$ we have $\Delta(a) = v \otimes a + a \otimes w$ and $\epsilon(a)=0$, that is, arrows are skew-primitive elements.

For every natural number $i$, we define the {\em $ext^i$-quiver} $\overrightarrow{\wis{ext}}^i_R$ to have one vertex $v_S$ for every $S \in \wis{simp}(R)$ and such that the number of arrows from $v_S$ to $v_T$ is equal to the dimension of the space $Ext^i_R(S,T)$. With $\wis{ext}^i_R$ we denote the $\C$-vectorspace on the arrows of $\overrightarrow{\wis{ext}}^i_R$.

The {\em Yoneda-space} $\wis{ext}^{\bullet}_R = \oplus \wis{ext}^i_R$ is endowed with a natural $A_{\infty}$-structure \cite{Keller}, defining a linear map (the {\em homotopy Maurer-Cartan map}, \cite{Segal})
\[
\mu = \oplus_i m_i~:~\C \overrightarrow{\wis{ext}}^1_R \rTo \wis{ext}^2_R \]
from the path coalgebra $\C \overrightarrow{\wis{ext}}^1_R$ of the $ext^1$-quiver to the vectorspace $\wis{ext}^2_R$, see \cite[\S 2.2]{Keller} and \cite{Segal}.

\begin{theorem} The dual coalgebra $R^o$ is Takeuchi-equivalent to the pointed coalgebra $R^{\dagger}$ which is  the sum of all subcoalgebras contained in the kernel of the linear map 
\[
\mu = \oplus_i m_i~:~\C \overrightarrow{\wis{ext}}^1_R \rTo \wis{ext}^2_R \]
 determined by the $A_{\infty}$-structure on the Yoneda-space $\wis{ext}^{\bullet}_R$.
\end{theorem}

\begin{proof}
We can reduce to finite subquivers as any subcoalgebra is the limit of finite dimensional subcoalgebras and because any finite dimensional $R$-representation involves only finitely many simples. Hence, the statement is a global version of the result on finite dimensional algebras due to B. Keller \cite[\S 2.2]{Keller}. 

Alternatively, we can use the results of E. Segal \cite{Segal}. Let $S_1,\hdots,S_r$ be distinct simple finite dimensional $R$-representations and consider the semi-simple module $M=S_1 \oplus \hdots \oplus S_r$ which determines an algebra epimorphism
\[
\pi_M~:~R \rOnto M_{n_1}(\C) \oplus \hdots \oplus M_{n_r}(\C) = B \]
If $\mathfrak{m} = Ker(\pi_M)$, then the $\mathfrak{m}$-adic completion $\hat{R}_{\mathfrak{m}} = \underset{\leftarrow}{lim}~R/\mathfrak{m}^n$ is an augmented $B$-algebra and we are done if we can describe its finite dimensional (nilpotent) representations. Again, consider the $A_{\infty}$-structure on the Yoneda-algebra $Ext^{\bullet}_R(M,M)$ and the quiver on $r$-vertices $\overrightarrow{\wis{ext}}^1_R(M,M)$ and the homotopy Mauer-Cartan map
\[
\mu_M = \oplus_i m_i~:~\C \overrightarrow{\wis{ext}}^1_R(M,M) \rTo Ext^2_R(M,M) \]
Dualizing we get a subspace $Im(\mu_M^*)$ in the path-{\em algebra} $\C \overrightarrow{\wis{ext}}^1_R(M,M)^*$ of the dual quiver. Ed Segal's main result \cite[Thm 2.12]{Segal} now asserts that $\hat{R}_{\mathfrak{m}}$ is Morita-equivalent to
\[
\hat{R}_{\mathfrak{m}} \underset{M}{\sim} \frac{(\C \overrightarrow{\wis{ext}}^1_R(M,M)^*)^{\hat{}}}{(Im(\mu_M^*))} \]
where $(\C \overrightarrow{\wis{ext}}^1_R(M,M)^*)^{\hat{}}$ is the completion of the path-algebra at the ideals generated by the paths of positive length. The statement above is the dual coalgebra version of this.
\end{proof}

\noindent
{\bf Proof of theorem~\ref{reconstructbrane} : } The morphism between the thin noncommutative schemes
\[
\beta^* : \wis{rep}_A \rTo \wis{rep}_R \]
corresponds to the coalgebra-map dual to $\beta$ 
\[
\beta^o : A^o \rTo R^o \]
Dualizing this map again we obtain a $\C$-algebra map, and composing with the natural map $R \rTo R^{o \ast}$ and the observations of example~\ref{exampleAzu} we obtain an algebra map
\[
R \rTo R^{o \ast} \rTo^{\beta^{o \ast}}  A^{o \ast} = \prod_{\mathfrak{m}} M_n(\hat{C}_{\mathfrak{m}}) \]
the components of which are the maps from the global sections of the structure sheaf of $R$ in the \'etale topology to the stalks of the structure sheaf of $A$ in the \'etale topology, induced by $\beta$. By \'etale descent we can therefore reconstruct the algebraic D-brane $\beta : R \rTo A$ from them.


\begin{thebibliography}{10}

 \bibitem{Chin}
 William Chin, {\it A brief introduction to coalgebra representation theory}, in "Hopf Algebras" M. Dekker Lect. Notes in Pure and Appl. Math. (2004) 109-132. Online at http://condor.depaul.edu/~wchin/crt.pdf


\bibitem{deJong}
Aise Johan de Jong, \href{http://www.math.columbia.edu/algebraic_geometry/stacks-git/}{The Stacks Project}, {\it Algebraic stacks - Examples}



\bibitem{DeMeyerIngraham}
Frank De Meyer and E. Ingraham,
{\it Separable algebras over commutative rings}, Springer LNM {\bf 181} (1970)

 \bibitem{Keller}
 Bernhard Keller, {\it A-infinity algebras in representation theory}, Contribution to the Proceedings of ICRA IX, Beijing (2000). Online at http://www.math.jussieu.fr/~keller/publ/art.dvi



\bibitem{KnusOjanguren}
Max-Albert Knus and Manuel Ojanguren,
{\it Th\'eorie de la Descente et Alg\`ebres d'Azumaya}, Springer LNM {\bf 389} (1974)


\bibitem{KontSoib}
Maxim Kontsevich and Yan Soibelman, {\it Notes on $A_{\infty}$-algebras, $A_{\infty}$-categories and non-commutative geometry I}, arXiv:math.RA/0606241 (2006)


\bibitem{LBbook}
Lieven Le Bruyn, {\it Noncommutative geometry and Cayley-smooth orders}, Pure and applied mathematics {\bf 290}, Chapman \& Hall (2008)

\bibitem{LBSymens}
Lieven Le Bruyn and Stijn Symens,
{\it Partial desingularizations arising from non-commutative algebras}, 
\href{http://arxiv.org/abs/0507494}{arXiv:0507494} (2005)


 
\bibitem{LY1} Chien-Hao Liu and Shing-Tung Yau, {\it Azumaya-type noncommutative spaces and morphisms therefrom : Polchinski's D-branes in string theory from Grothendieck's viewpoint}, \href{http://arxiv.org/abs/0709.1515}{arXiv:0709.1515} (2007)

 \bibitem{LY2} Si Li, Chien-Hao Liu, Ruifang Song and Shing-Tung Yau, {\it Morphisms from Azumaya prestable curves with a fundamental module to a projective variety: Topological D-strings as a master object for curves}, \href{http://arxiv.org/abs/0809.2121}{arXiv:0809.2121} (2008)
  
\bibitem{LY3} Chien-Hao Liu and Shing-Tung Yau, {\it Azumaya structure on D-branes and resolution of ADE orbifold singularities revisited: Douglas-Moore vs. Polchinski-Grothendieck}, \href{http://arxiv.org/abs/0901.0342}{arXiv:0901.0342} (2009)

\bibitem{LY4} Chien-Hao Liu and Shing-Tung Yau,  {\it Azumaya structure on D-branes and deformations and resolutions of a conifold revisited: Klebanov-Strassler-Witten vs. Polchinski-Grothendieck}, \href{http://arxiv.org/abs/0907.0268}{arXiv:0907.0268} (2009)

\bibitem{LY5} Chien-Hao Liu and Shing-Tung Yau, {\it Nontrivial Azumaya noncommutative schemes, morphisms therefrom, and their extension by the sheaf of algebras of differential operators: D-branes in a B-field background \`a la Polchinski-Grothendieck Ansatz}, \href{http://arxiv.org/abs/0909.2291}{arXiv:0909.2291} (2009)

\bibitem{LY6} Chien-Hao Liu and Shing-Tung Yau, {\it D-branes and Azumaya noncommutative geometry: From Polchinski to Grothendieck}, \href{http://arxiv.org/abs/1003.1178}{arXiv:1003.1178} (2010)


\bibitem{MacLane} Saunders Mac Lane, {\it Categories for the working mathematician}, Springer Graduate Texts in Math. 5, 2nd edition (1997)

\bibitem{Milne}
James S. Milne,
{\it \'Etale cohomology}, Princeton Math. Series {\bf 33} (1980)

\bibitem{ProcesiBook}
Claudio Procesi, {\it Rings with polynomial identities}, Pure and Appl. Math. {\bf 17} Marcel Dekker (1973)

\bibitem{ProcesiCH}
Claudio Procesi, {\it A formal inverse to the Cayley-Hamilton theorem}, J. Alg. {\bf 107} (1987) 63-74

\bibitem{Procesi1}
Claudio Procesi,
{\it  Deformations of representations},
Methods in ring theory (Levico Terme, 1997), 247--276, Lecture Notes in Pure and Appl. Math. {\bf 198,} Marcel Dekker  (1998) 

\bibitem{Segal}
Ed Segal, {\it The $A_{\infty}$ deformation theory of a point and the derived category of local Calabi-Yaus}, math.AG/0702539 (2007)

\bibitem{Sweedler}
Moss E. Sweedler, {\it Hopf Algebras}, monograph, W.A. Benjamin (New York) (1969)

\bibitem{Takeuchi}
M. Takeuchi, {\it Morita theorems for categories of comodules}, J. Fac. Sci. Univ. Tokyo 24 (1977) 629-644


\bibitem{LNM887}
Fred Van Oystaeyen and Alain Verschoren,
{\it Noncommutative algebraic geometry}, Springer LNM {\bf 887} (1981)

\bibitem{FredBook}
Fred Van Oystaeyen,
{\it Algebraic geometry for associative algebras}, Lecture Notes in Pure and Appl. Math. {\bf 232} Marcel Dekker (2000)



\end{thebibliography}
\end{document}